\documentclass[12pt,reqno]{amsart}
\usepackage{hyperref,cleveref,amssymb,eqnarray}

\theoremstyle{plain}
\newtheorem{theorem}{Theorem}[section]
\newtheorem{proposition}[theorem]{Proposition}
\newtheorem{corollary}[theorem]{Corollary}
\newtheorem{lemma}[theorem]{Lemma}

\theoremstyle{definition}

\newtheorem{question}[theorem]{Question}
\newtheorem{example}[theorem]{Example}

\newcommand{\term}[1]{{\textit{\textbf{#1}}}}   

\newcommand{\Pelc}[1]{\left(\displaystyle{\oplus_{n=1}^\infty} #1_n\right)_p}

\DeclareSymbolFont{bbold}{U}{bbold}{m}{n}
\DeclareSymbolFontAlphabet{\mathbbold}{bbold}

\DeclareMathOperator{\diag}{diag}

\renewcommand{\le}{\leqslant}
\renewcommand{\ge}{\geqslant}

\begin{document}

\title[Commutators of positive operators]{Positive operators as commutators of positive operators\\
\vskip.2cm
\tiny{D\lowercase{edicated to the memory of}  J\lowercase{aroslav} Z\lowercase{em{\'a}nek}}}

\author{R. Drnov\v sek}
\address{Faculty of Mathematics and Physics, University of Ljubljana,
  Jadranska 19, 1000 Ljubljana, Slovenia}
\email{roman.drnovsek@fmf.uni-lj.si}

\author{M. Kandi\'c}
\address{Faculty of Mathematics and Physics, University of Ljubljana,
  Jadranska 19, 1000 Ljubljana, Slovenia}
\email{marko.kandic@fmf.uni-lj.si}


\thanks{The authors acknowledge the financial support from the
  Slovenian Research Agency (research core funding No. P1-0222).}

\date{\today}

\begin{abstract}
It is known that a positive commutator $C=A B - B A$ between positive operators on a Banach lattice is quasinilpotent whenever at least one of $A$ and $B$ is compact. In this paper we study the question under which conditions a positive operator can be written as a commutator between positive operators.
As a special case of our main result we obtain that positive compact operators on order continuous Banach lattices which admit order Pelczy\'nski decomposition are commutators between positive operators.
Our main result is also applied in the setting of a separable infinite-dimensional Banach lattice $L^p(\mu)$ $(1<p<\infty)$.
\end{abstract}

\maketitle

\section{Introduction}

Given an  associative algebra $\mathcal A$, the natural question is to determine all commutators of $\mathcal A$. Shoda \cite{Shoda} proved that a matrix $C \in \mathbb M_n(F)$ is a commutator if and only if the trace of $C$ is zero. Wintner \cite{Wintner} proved that the identity in a unital Banach algebra is not a commutator. By passing to the Calkin algebra, Wintner's result immediately implies that a bounded operator on a Banach space which is  of the form $\lambda I+K$ for some nonzero scalar $\lambda$ and a compact operator $K$ is not a commutator. The complete characterization of commutators in the Banach algebra $\mathcal B(\mathcal H)$ of all bounded operators on an infinite-dimensional Hilbert space $\mathcal H$ is due to Brown and Pearcy \cite{Brown:65}. They proved that a bounded operator $C$ on $\mathcal H$ is a commutator if and only if it is not of the form $\lambda I+K$ for some nonzero scalar $\lambda$ and some operator $K$ from the unique maximal ideal in $\mathcal B(\mathcal H)$. Apostol (\cite{Apostol:72,Apostol:73}) proved that a bounded operator on either $l^p$ ($1<p<\infty$) or $c_0$ is a commutator if and only if it is not of the form $\lambda I+K$ where $\lambda\neq 0$ and $K$ is compact. In the case of the Banach space $l^1$ the same characterization was obtained by Dosev in \cite{Dosev:09}. In the case of the Banach space $l^\infty$ Dosev and Johnson \cite{DJ:10} proved that a bounded operator is a commutator if and only if it is not of the form $\lambda I+K$ where $\lambda\neq 0$ and $K$ is strictly singular.

Our interest for commutators goes into a different direction. In this paper we are concerned with those positive operators on Banach lattices which can be written as commutators between positive operators.
The study of positive commutators of positive operators was initiated in \cite{Bracic:10}.
The assumption on positivity of $A$ and $B$ may lead to some restrictions on $C = A B - B A$.
If a positive operator $C$ can be written as a commutator of positive operators with one of them compact, then $C$ is necessarily quasinilpotent (\cite{Drnovsek:12,Gao:14}). In the finite-dimensional case, positive commutators of positive operators are always nilpotent.

In \Cref{Nilpotents} we first deal with nilpotent commutators. Under some mild assumptions which are satisfied for $L^p$-spaces $(1\leq p<\infty)$ we prove that a positive nilpotent (and compact) operator $C$ is always a commutator between a positive central operator and a positive nilpotent (and compact) operator. This result immediately implies that a positive matrix is nilpotent if and only if it can be written as a commutator between positive matrices. This characterization can be understood as
a ``positive" version of Shoda's characterization of matrices which are commutators.

In \Cref{Quasinilpotents} we consider quasinilpotent operators on $l^p$-spaces. If the entries of the matrix corresponding to a positive operator $C$ go sufficiently fast to zero, then $C$ can be written as a commutator between a positive diagonal operator and a positive compact quasinilpotent operator if and only if $C$ is quasinilpotent.

\Cref{compacts} is devoted to positive operators on Banach lattices which admit order Pelczy\'nski decomposition.
Our main result (\Cref{positive A B}) is inspired by the theorem of \cite[Theorem 3.3]{Schneeberger:71}
asserting that every compact operator on an infinite-dimensional separable Banach space $L^p(\mu)$ $ (1 < p < \infty)$
is a commutator between bounded operators. As a special case of our result we conclude that every positive compact operator
on an infinite-dimensional separable Banach lattice $L^p(\mu)$ $(1\leq p < \infty)$ is a commutator between positive operators.

\vspace{5mm}
\section{Preliminaries}

Pick $n \in\mathbb N$ and $1\leq k\leq n$. An upper-triangular matrix $A\in \mathbb M_n(\mathbb C)$ is said to be $k$-\term{super upper-triangular} whenever the diagonal of $A$ together with  first $(k-1)$-super diagonals of $A$ are zero. With respect to the above definition strictly upper-triangular matrices are precisely $1$-super upper-triangular and the zero matrix is $n$-super upper-triangular. An operator $T$ is said to be \term{nilpotent} if $T^m = 0$ for some
positive integer $m$. The smallest $m$ with this property is called the \term{nilpotency index} of $T.$ The \term{Jordan block} $J_n$ is a $n\times n$-matrix with zero entries everywhere except on the first super diagonal where all entries are equal to one.

Suppose now that $X$ is a vector lattice and $T:X\to X$ is a positive operator.
The \term{absolute kernel} of the operator $T$ is defined by $\mathcal N(T) = \{x\in X:\; T|x|=0\}$.
It should be noted that $\mathcal N(T)$ is always an ideal of $X$. If $T$ is also order continuous, then $\mathcal N(T)$ is a band in $X$. A band $B$ in $X$ is a \term{projection band} if $X=B\oplus B^d.$ The band projection $P_B$ of $X$ to $B$ is an order continuous positive operator. A positive operator $T$ on $X$ is said to be \term{central} whenever there exists a positive real number $\lambda$ such that $0\leq T\leq \lambda I$. In the case when $X=\mathbb R^n$, a positive matrix $T$ is central if and only if it is diagonal.
Positive operators $A$ and $B$ on a vector lattice $X$ are said to \term{semi-commute} if either $AB\geq BA$ or $BA\geq AB.$

%
%

Suppose $X$ is a Banach lattice. The closed unit ball of $X$ is denoted by $B_X$. Banach lattice $X$ is said to be \term{order continuous} whenever $x_\alpha\downarrow 0$ implies $\|x_\alpha\| \downarrow 0.$ A set $A$ of $X$ is said to be \term{almost order bounded} if for every $\epsilon>0$ there exists $u\geq 0$ in $X$ such that $A\subseteq [-u,u]+\epsilon B_X.$  By \cite[Theorem 122.1]{Zaanen:83}, a set $A$ is almost order bounded if and only if for each $\epsilon>0$ there exists $u\in X^+$ such that $\|(|x|-u)^+\|<\epsilon$ for each $x\in A.$ If $X$ has order continuous norm, then order intervals are weakly compact, so that by Grothendieck's theorem (see e.g. \cite[Theorem 3.44]{Aliprantis:06}) almost order bounded sets are relatively weakly compact.
On the other hand, order continuity of the norm on $X$ (resp., $X^*$) can be characterized through an approximation property which can be understood as a ``local almost order boundedness" of the closed unit ball of $X^*$ (resp., $X$). The following approximation theorem is due to Dodds and Fremlin (see e.g. \cite{Dodds:79} and \cite{Aliprantis:06}).

\begin{theorem}\label{Dodds-Fremlin}
For a Banach lattice $X$ the following assertions hold.
\begin{enumerate}
\item[(a)] $X$ has order continuous norm if and only if for each $\epsilon>0$ and each $x\in X^+$ there exists some $0\leq \psi\in X^*$ such that $(|\varphi|-\psi)^+(x)<\epsilon$ for each $\varphi \in B_{X^*}.$
\item[(b)] $X^*$ has order continuous norm if and only if for each $\epsilon>0$ and each $0\leq \psi\in X^*$ there exists some $y\in X^+$ such that $\psi((|x|-y)^+)<\epsilon$ for each $x\in B_X.$
\end{enumerate}
\end{theorem}

 It should be mentioned here that, in general, if $X$ has order continuous norm, the closed unit ball of $X^*$ is not necessarily almost order bounded. As an example consider the Banach lattice $c_0$.

An operator $T:X\to X$ is said to be \term{semi-compact} if $T(B_X)$ is almost order bounded in $X$. By the discussion above, it immediately follows that every almost order bounded operator on a Banach lattice with order continuous norm is weakly compact. Recall that $X$ has the \term{positive Schur property} whenever $0\leq x_n\to 0$ weakly implies $x_n\to 0$. If $X$ has the positive Schur property, then every relatively weakly compact in $X$ is almost order bounded (\cite[Theorem 3.14]{GX14}). Hence, when $X$ has the positive Schur property, classes of weakly compact and semi-compact operators coincide. In particular, the latter applies in the case of $L^1$-spaces. An order bounded operator $T$ on $X$ is said to be \term{AM-compact} whenever it maps order bounded sets to relatively compact sets.

%

\vspace{5mm}
\section{Nilpotent commutators} \label{Nilpotents}

If a positive operator $C$ can be written as a commutator $AB-BA$ of positive operators $A$ and $B$ with one of them nilpotent, then \cite[Lemma 2.1]{Drnovsek:11} implies that $C$ is nilpotent.
The following theorem shows that, under mild technical assumptions, the converse also holds.

\begin{theorem}\label{nilpotent positive operator}
Let $C$ be an order continuous positive operator on a Banach lattice $X$ with the projection property.
Then $C$ is nilpotent if and only if there exist a positive central operator  $A$ and a positive order continuous nilpotent operator $B$ such that  $C=AB-BA.$ Furthermore, if $C$ is compact, then also $B$ can be chosen to be compact.
\end{theorem}

\begin {proof}
We will prove the theorem by induction on the nilpotency index of the operator $C$.
If $C=0$, then $A=I$ and $B=0$ meet the requirements of the theorem.

Suppose now that $C$ is nonzero and $C^2=0$. Since $C$ is order continuous and $X$ has the projection property, the absolute kernel $\mathcal N(C)$ is a projection band in $X$. With respect to the band decomposition $X=\mathcal N(C)\oplus \mathcal N(C)^d$, the operator $C$ can be written as a $2\times 2$ block operator matrix
$$C=\left[\begin {array}{cc}
0 & D\\
0 & E
\end {array}
\right].$$

We claim that $E=0$. In order to prove this, pick an arbitrary positive vector $x\in \mathcal N(E)^d$.  From $C^2=0$ we conclude $DE=0$ and $E^2 = 0$, and so $C\left[\begin{array}{c}
0 \\ Ex
\end{array}\right]=0.$ Therefore, $\left[\begin{array}{c}
0 \\ Ex
\end{array}\right] \in \mathcal N(C)\cap \mathcal N(C)^d=\{0\}.$
This proves $E=0$.

If we define
$$A=\left[\begin {array}{cc}
I&0\\
0 & 0
\end {array}
\right] \qquad \textrm{and}
\qquad
B=C=\left[\begin {array}{cc}
0&D\\
0 & 0
\end {array}
\right],$$ we have $C=AB-BA.$ It is obvious that $A$ is central and $B$ is nilpotent and order continuous.

Suppose now that the theorem holds for every order continuous nilpotent operator with its nilpotency index at most $n$ and choose an order continuous positive operator $C$ with $C^{n+1}=0$ and $C^n\neq 0$.
Then $X$ can be decomposed as a band direct sum
\begin{align}\label{razcep}
X=X_1\oplus \cdots \oplus X_{n+1}
\end{align}
where $X_i=\mathcal N(C^{i})\cap \mathcal N(C^{i-1})^d$ for $i=1,\ldots,n+1.$
With respect to the decomposition (\ref{razcep}) the operator  $C$ can be written as a $(n+1)\times (n+1)$-upper-triangular block operator matrix
$$C=
\left[
\begin {array}{cccccc}
0 & C_{1,2} & \cdots & C_{1,n+1} \\
 & \ddots  & \ddots & \vdots\\
& & \ddots & C_{n,n+1}\\
& &  & 0
\end {array}
\right]
.$$
Let $\widetilde C$ be the compression of $C$ to the band $X_1^d.$ Then $\widetilde C^n=0$, so that by the induction hypothesis there exist a positive central operator $\widetilde A$ and a positive order continuous nilpotent operator $\widetilde B$ on $X_1^d$ such that
$\widetilde C=\widetilde A\widetilde B-\widetilde B\widetilde A.$
If $C$ is compact, then $\widetilde C$ is compact, so that by the induction hypothesis the operator $\widetilde B$ can also be chosen to be compact.

Let us choose a positive number $\alpha$ greater than the spectral radius of $\widetilde{A}$. Then
$$(\alpha I-\widetilde A)^{-1}=\sum_{n=0}^\infty \frac{\widetilde{A}^n}{\alpha^{n+1}}$$
is a positive central operator on $X_1^d.$ Let $P$ be the band projection onto $X_1.$
If we define the operator $B' : X_1^d \to X_1$ by
$$ B'=PC(I-P)(\alpha I-\widetilde A)^{-1}, $$
the operator
$$B=\left[\begin {array}{cc}
0 & B'\\
0 &\widetilde B
\end {array}\right]$$ is order continuous, positive and nilpotent.
If $C$ is compact, then $B'$ is also compact, so that $B$ is compact.
Since $\widetilde A$ is central on $X_1^d$,
the operator
$$A=\left[\begin {array}{cc}
\alpha I & 0\\
0 &\widetilde A
\end {array}\right]$$ is central on $X$.
We finish the proof by an easy calculation
$$AB-BA=\left[\begin {array}{cc}
0 & \alpha B'-B'\widetilde A\\
0 & \widetilde A\widetilde B-\widetilde B\widetilde A
\end {array}\right]
=\left[\begin {array}{cc}
0 & PC(I-P)\\
0 & \widetilde A\widetilde B-\widetilde B\widetilde A
\end {array}\right]=C . $$
\end {proof}

As an immediate consequence of \Cref{nilpotent positive operator}, we obtain the following corollary
that can be considered as an order analog of Shoda's result.

\begin{corollary}
A positive matrix $C$ can be written as a commutator of positive matrices $A$ and $B$ if and only if $C$ is nilpotent.
Moreover, we can choose $A$ to be diagonal and $B$ to be permutation similar to a strictly upper-triangular matrix.
\end{corollary}

In the rest of this section we characterize positive commutators of positive nilpotent matrices.
We first prove a necessary condition in more general setting.

\begin{proposition} \label{two_super}
Let $X$ be at least $2$-dimensional Banach lattice with the projection property, and let $A$ and $B$ be positive
nilpotent order continuous operators on $X$. If  the commutator $AB-BA$ is positive, then there exists a finite band decomposition
$$ X= X_1 \oplus \cdots \oplus X_n $$
with respect to which it has a $2$-super upper-triangular matrix form.
\end{proposition}

\begin{proof}
Since $AB-BA\geq 0$,
\cite[Lemma 2.2]{Drnovsek:11} implies that $D:=A+B$ is nilpotent as well. Let $n$ be the nilpotency index of $D$.

If $n=1$, then $D=0$ and hence $A=B=0$. Therefore, $AB-BA$ is $2$-super strictly upper-triangular for any band decomposition $X=X_1 \oplus X_2$.

If $n>1$, then we decompose $X$ as $X=X_1 \oplus \cdots \oplus X_n$ where $X_i=\mathcal N(D^{i})\cap \mathcal N(D^{i-1})^d$ for $i=1,\ldots,n.$ With respect to this decomposition, we have
$$ D= \left[
\begin {array}{cccccc}
0 & D_{1,2} & \cdots & D_{1,n} \\
 & \ddots  & \ddots & \vdots\\
& & \ddots & D_{n-1,n}\\
& &  & 0
\end {array}
\right]  .  $$
Since $0\leq A,B \leq D$, we also have
$$ A=
\left[
\begin {array}{cccc}
0 & A_{1,2} & \cdots & A_{1,n} \\
 & \ddots  & \ddots & \vdots\\
& & \ddots & A_{n-1,n}\\
& &  & 0
\end {array}
\right] \qquad \! \! \! \textrm{and} \! \! \!
\qquad
B=
\left[
\begin {array}{cccc}
0 & B_{1,2} & \cdots & B_{1,n} \\
& \ddots  & \ddots & \vdots\\
& & \ddots & B_{n-1,n}\\
& &  & 0
\end {array}
\right]   .  $$
Now it is obvious that $C=AB-BA$ is $2$-super strictly upper-triangular.
\end{proof}

We now show the converse assertion in the case of matrices.

\begin{proposition}\label{jordan}
Suppose that $n \geq 3$ and $2 \leq k \leq n$. Then every  $k$-super upper-triangular positive matrix
$C \in \mathbb M_n(\mathbb C)$ can be written as $AB-BA$, where $A$ is a positive nilpotent matrix
and $B=J_n^{k-1}$.
\end{proposition}

\begin{proof}
Define the positive strictly upper-triangular matrix by
$$ A = C \! \cdot \! (J_n^T)^{k-1} + J_n^{k-1} \! \cdot \! C \! \cdot \! (J_n^T)^{2(k-1)} +
J_n^{2(k-1)} \! \cdot \! C \! \cdot \! (J_n^T)^{3(k-1)} + \ldots  . $$
Note that the number of sumands is finite, as $J_n$ is nilpotent.
Then we have
$$ A B = C + J_n^{k-1} \! \cdot \! C \! \cdot \! (J_n^T)^{k-1} +
J_n^{2(k-1)} \! \cdot \! C \! \cdot \! (J_n^T)^{2(k-1)} + \ldots  = C + B A , $$
since $S \cdot (J_n^T)^{k-1} \cdot J_n^{k-1} = S$ for each matrix $S$ having the first $(k-1)$ columns
equal zero.
\end{proof}

The following characterization follows easily from Propositions  \ref{two_super} and \ref{jordan}.

\begin{corollary}
A positive matrix $C \in \mathbb M_n(\mathbb R)$ is permutation similar to a $2$-super strictly upper-triangular matrix
if and only if there exist positive nilpotent matrices $A$ and $B$ such that $C=AB-BA$.
\end{corollary}

\vspace{5mm}
\section{Quasinilpotent commutators}\label{Quasinilpotents}

If a positive operator $C$ on a Banach lattice is a commutator of two positive operators with at least one of them
compact, then $C$ is a quasinilpotent operator, by \cite{Drnovsek:12} or \cite{Gao:14}.
The following theorem provides conditions on $C$ under which the converse assertion is also true.

\begin {theorem}
Let $C = (c_{i,j})_{i,j=1}^{\infty}$ be a positive operator on the space $l^p$ ($1\le p \le \infty$) such that
$\sum_{i=1}^{\infty} \sum_{j=1}^{\infty} \sqrt{c_{i,j}} < \infty$. Then the following assertions are equivalent:
\begin{itemize}

\item[(a)] $C= A B - B A$  for a diagonal positive operator $A$ on $l^p$ and
a quasinilpotent compact positive operator $B$ on $l^p$;

\item[(b)] $C= A B - B A$  for some positive operators $A$ and $B$  on $l^p$ with at least one of them compact;

\item[(c)] $C$ is a quasinilpotent operator.

\end{itemize}

\end {theorem}

\begin {proof}
The implication (a) $\Rightarrow$ (b) is clear, and the implication (b) $\Rightarrow$ (c) holds by \cite{Drnovsek:12} or \cite{Gao:14}.

Let us prove that (c) implies (a). It follows from the considerations in \cite[Section 2]{MMR09}
that there exists a total order $\succeq$ on $\mathbb N$ such that $c_{i,j} = 0$ whenever $i \succeq j$.
Define
$$ d_i = \sum_{j=1}^{\infty}  \sqrt{c_{i,j}} = \sum_{j \succ i} \sqrt{c_{i,j}} \ \ \ \textrm{and} \ \
a_k = \sum_{i \succeq k} d_i . $$
Then $\sum_{i=1}^{\infty} d_i = \sum_{i=1}^{\infty} \sum_{j=1}^{\infty} \sqrt{c_{i,j}} < \infty$ and
$d_i \ge \sqrt{c_{i,j}}$ for all $i$ and $j$.
Let $A$ be a diagonal matrix with diagonal entries $(a_i)_{i=1}^{\infty}$, and
let $B = (b_{i,j})_{i,j=1}^{\infty}$ be a matrix defined by
$b_{i,j} = 0$ if $c_{i,j} = 0$, and
$$ b_{i,j} = {c_{i,j} \over \sum_{i \preceq k \prec j} d_k}  \textrm{     otherwise.} $$
Since $\sum_{i=1}^{\infty} d_i < \infty$, $A$ defines a bounded positive operator on $l^p$. Since
$$ b_{i,j} \leq {c_{i,j} \over d_i} \le {c_{i,j} \over \sqrt{c_{i,j}}} = \sqrt{c_{i,j}} , $$
we have
$$ \sum_{i=1}^{\infty} \sum_{j=1}^{\infty} b_{i,j} \leq
    \sum_{i=1}^{\infty} \sum_{j=1}^{\infty} \sqrt{c_{i,j}} < \infty , $$
and so $B$ defines a bounded operator on $l^1$ and on $l^\infty$ with the corresponding norms
$\|B\|_1$ and $\|B\|_\infty$ less than or equal to $\sum_{i=1}^{\infty} \sum_{j=1}^{\infty} b_{i,j}$.
Moreover, in both cases the operator $B$ is compact as it can be written as a limit of a sequence of finite-rank operators.
It follows from \cite[Theorem 1.6.1]{Davies:90} that $B$ is also a compact operator on $l^p$ $(1<p<\infty)$.
Finally, Ringrose's Theorem (see e.g. \cite[Theorem 7.2.3]{RR00}) implies that $B$ is also quasinilpotent.

Given $i$ and $j$ with $i \prec j$, the $(i,j)$ entry of $AB-BA$ is equal to
$$ a_i b_{i,j} - b_{i,j} a_j = b_{i,j} (a_i - a_j) = b_{i,j}  \cdot \sum_{i \preceq k \prec j} d_k = c_{i,j} . $$
If $i \succeq j$, then $b_{i,j} = c_{i,j} = 0$.  It follows that  $C= A B - B A$, and so the proof is complete.
\end {proof}

There is a quasinilpotent compact positive operator on $l^2$ that cannot be written as a commutator of a diagonal positive
operator and  a positive operator on $l^2$. Furthermore, it cannot be written as a commutator of a positive
operator and a quasinilpotent positive operator on $l^2$ with one of them compact. For the construction of our example we need the following lemma.

\begin{lemma}\label{nicelni diagonalci}
If $S=(s_{i,j})_{i,j=1}^\infty:l^2\to l^2$ is a quasinilpotent positive operator,
then $s_{k,k}=0$ for each $k\in \mathbb N.$
\end{lemma}

\begin{proof}
Since $S$ is positive,  we first conclude
$$0 \leq s_{k,k}^n\leq \langle S^ne_{k},e_k\rangle\leq \|S^n\|$$
for all $k,n\in\mathbb N$, where $\{e_k\}_{k=1}^\infty$ is the standard basis of $l^2$.
Now, quasinilpotency of $S$ implies that $s_{k,k}=0$ for each $k\in\mathbb N$.
\end{proof}

\begin{example}\label{shift}
Let $\{w_i\}_{i=1}^{\infty}$ be a decreasing
sequence of positive real numbers converging to $0$ slowly enough that
the series $\sum_{i=1}^{\infty} w_i$ is divergent. Then the matrix
$$ C= \left[
\begin {array}{ccccc}
0 & w_1 & 0 & 0  & \cdots \\
0 & 0 & w_2 & 0  & \cdots \\
0 & 0 & 0 &  w_3  & \cdots \\
\vdots & \vdots & \vdots & \ddots & \ddots
\end {array}
\right]  $$
defines a weighted shift on the space $l^2$ that is a quasinilpotent compact positive operator.

$\bullet$ \emph{The operator $C$ cannot be written as a commutator between a positive diagonal operator and a positive operator.}

Assume that there exist a diagonal positive operator $A = \diag(a_1, a_2, \ldots)$ and
a positive operator $B = (b_{i,j})_{i,j=1}^{\infty}$ on $l^2$ such that $C = A B - B A$.
Then  $w_i = a_i b_{i,i+1} - b_{i,i+1} a_{i+1} = b_{i, i+1} (a_i - a_{i+1})$ for all $i \in \mathbb N$.
In particular, we have $a_i > a_{i+1}$ for all $i \in \mathbb N$. Therefore,  for any $n \in \mathbb N$, we have
$$ \sum_{i=1}^{n} w_i = \sum_{i=1}^{n} b_{i, i+1} (a_i - a_{i+1}) \le
\|B\| \cdot \sum_{i=1}^{n} (a_i - a_{i+1}) \le \|B\| \cdot a_1 , $$
which contradicts the condition that $\sum_{i=1}^{\infty} w_i$ is a divergent series. A similar argument shows that $C$ cannot be written as $AB-BA$ with $A$ positive and $B$ positive and diagonal.

$\bullet$ \emph{The operator $C$ cannot be written as a commutator between two positive operators where at least one of them is compact and at least one of them is quasinilpotent.}

Assume that there exist positive operators $A =  (a_{i,j})_{i,j=1}^{\infty}$ and
$B = (b_{i,j})_{i,j=1}^{\infty}$ on $l^2$  such that $C = A B - B A$. Suppose that at least one of $A$ and $B$ is compact and suppose also that at least one of them is quasinilpotent. Here we do not exclude the possibility that one of them is simultaneously compact and quasinilpotent.

We claim that the corresponding matrices of $A$ and $B$ are necessarily upper-triangular. In order to prove this, we define $D:=A+B$ and observe that $D$ semi-commutes with $A$ and $B$.
If $D$ is ideal-irreducible, then \cite[Corollary 3.5]{Gao:14} yields  $C=AB-BA=0$ which is a contradiction.
Hence $D$ is ideal-reducible. An application of the Ideal-triangularization lemma (see e.g. \cite{Drnovsek:09}) implies that $D$ is ideal-triangularizable. Since every closed ideal that is invariant under $D$ is also invariant under both $A$ and $B$, it is also invariant under $C$. However, closed ideals invariant under $C$ are precisely those which are of the form
$\textrm{span}\{e_1,\ldots,e_n\}$ for some $n$. This proves the claim.

If $A$ is quasinilpotent, then the matrix corresponding to the operator $A$ is strictly upper-triangular by \Cref{nicelni diagonalci}.
From $C=AB-BA$ it follows
$$w_i=a_{i,i+1}b_{i+1,i+1}-a_{i,i+1}b_{i,i}=a_{i,i+1}(b_{i+1,i+1}-b_{i,i}),$$ for all $i\in\mathbb N$
so that $b_{i+1,i+1}>b_{i,i}.$  Then
\begin{align*}
\sum_{i=1}^{n} w_i& = \sum_{i=1}^{n} a_{i, i+1} (b_{i+1, i+1} - b_{i,i}) \le
\|A\| \cdot  \sum_{i=1}^{n} (b_{i+1, i+1} - b_{i,i})\\
 &= \|A\|(b_{n+1,n+1}-b_{1,1})\leq  \|A\| \cdot \|B\|,
\end{align*}
which again contradicts the condition that $\sum_{i=1}^{\infty} w_i$ is a divergent series.
The case when $B$ is quasinilpotent can be treated similarly.
%

%
\end{example}

The following question still remains open.

\begin{question}
Is it possible to express the operator $C$ from \Cref{shift} as a commutator $AB-BA$ of positive operators $A$ and $B$ with at least one of them compact?
\end{question}

Observe that the operator $C$ from \Cref{shift} can be written as a commutator of positive operators, by  \Cref{positive A B} in the next section.

\vspace{5mm}
\section{Compact commutators}\label{compacts}

The $l^p$-\term{direct sum} ($1\le p < \infty$) of Banach spaces (resp., Banach lattices) $X_1, X_2, \ldots$ is
the Banach space (resp., Banach lattice)
$$ \left(\displaystyle{\oplus_{n=1}^\infty} X_n\right)_p = \left\{x = (x_1, x_2, \ldots) : x_n \in X_n,
\|x\|^p = \sum\limits_{n=1}^\infty \|x_n\|^p < \infty \right\} . $$
Similarly, the $l^\infty$-\term{direct sum} is defined by
$$ \left(\displaystyle{\oplus_{n=1}^\infty} X_n\right)_\infty = \left\{x = (x_1, x_2, \ldots) : x_n \in X_n,
\|x\| = \sup_{n \in \mathbb N} \|x_n\| < \infty \right\} . $$
A Banach space $X$ admits a \term{Pelczy\'nski decomposition} if
$X=\left(\displaystyle{\oplus_{n=1}^\infty} X_n\right)_p$, where  $1\le p \le \infty$ and
$X_n$ is isometric to $X$ for all $n$.
Let us introduce its order analog.
We say that a Banach lattice $X$ admits an \term{order Pelczy\'nski decomposition}
if $X=\left(\displaystyle{\oplus_{n=1}^\infty} X_n\right)_p$, where  $1\le p \le \infty$ and, for  each $n$,
$X_n$ is a band of $X$ isometric and order isomorphic to $X$.
In this case the Banach lattice $(X \oplus X)_p$ is isometric and order isomorphic to $X$,
and so (by induction) this is also true for any finite direct sums.

For the proof of \Cref{positive A B} which is the main result of this paper, we need the following proposition.
It should be noted that the assertion in (a) is a part of \cite[Theorem 126.3]{Zaanen:83}.

\begin{proposition}\label{monotoni normni pad}
Let $\{S_n\}_{n=1}^{\infty}$  and $T$ be positive operators on a Banach lattice $X$ such that
$T\geq S_n\downarrow 0$. Then $\|S_n\|\downarrow 0$ in either of the following cases.
\begin{enumerate}
\item [(a)] $T$ is semi-compact, and $X$ and $X^*$ have order continuous norms.
\item [(b)] $T$ is semi-compact, $T^*$ is AM-compact and $X$ has order continuous norm.
\item [(c)] $T$ is AM-compact, $T^*$ is semi-compact and $X^*$ has order continuous norm.
\end{enumerate}
\end{proposition}

In the proof we will use the identity
$w=(w-z)^++w\wedge z$ which holds for arbitrary vectors $w$ and $z$ of a given vector lattice.

\begin{proof}
As already observed, (a) follows from \cite[Theorem 126.3]{Zaanen:83}.

To see (b) pick $\epsilon>0$. Since $T$ is semi-compact, there exists $y\in X^+$
such that $\|(Tx-y)^+\|<\epsilon$ for all $0\leq x\in B_X.$
By \Cref{Dodds-Fremlin}(a), there exists $0\leq \psi \in X^*$ such that $(\varphi-\psi)^+(y)<\epsilon$ for each
$0\leq \varphi \in B_{X^*}.$ Therefore, for $0\leq x\in B_X$ and $0\leq \varphi\in B_{X^*}$ we have
\begin{align*}
\varphi(S_nx)&=\varphi((S_nx-y)^+)+\left[(\varphi-\psi)^++\varphi\wedge \psi\right](y\wedge S_nx)\\
&\leq \varphi((Tx-y)^+)+(\varphi-\psi)^+(y)+ \psi(S_nx)\\
&\leq 2\epsilon + \|S_n^*\psi\|.
\end{align*}
It remains to show that  $\|S_n^*\psi\| \downarrow 0$, because we then obtain that $\|S_n\| \downarrow 0$.
Since the norm of $X$ is order continuous, the functional $\psi$ is order continuous, so that $S_n\downarrow 0$ implies $S_n^*\psi\downarrow 0.$  On the other hand, the sequence
$\{S_n^*\psi\}_{n\in\mathbb N}$ is contained in the relatively compact interval $[0,T^*\psi]$. Hence, there is a convergent subsequence $\{S_{n_k}^* \psi\}_{k\in\mathbb N}$ of the sequence $\{S_n^* \psi\}_{n\in\mathbb N}.$ Since $S_{n_k}^*\psi\downarrow 0$, by \cite[Theorem 15.3]{Zaanen:97} we have
$\|S_{n_k}^*\psi \| \downarrow 0$.
Since $S_n^*\psi \downarrow 0$, we finally conclude that $\|S_n^*\psi\| \downarrow 0$.

The proof of (c) is very similar to the proof of (b). First choose $\epsilon>0$. By semi-compactness of $T^*$,
 we can find $0\leq \psi\in X^*$ with $\|(T^*\varphi-\psi)^+\|<\epsilon$ for each $0\leq \varphi\in B_{X^*}.$
By \Cref{Dodds-Fremlin}(b), there exists $0\leq y\in X$ with
$\psi((x-y)^+)<\epsilon$ for each $0\leq x\in B_X.$ As in (b), we obtain
\begin{align*}
(S_n^*\varphi)(x)&= (S_n^*\varphi-\psi)^+(x)+(S_n^*\varphi\wedge \psi)((x-y)^++x\wedge y)\\
&\le 2\epsilon+\|S_ny\|.
\end{align*}
As in (b), one can prove $\|S_ny\| \downarrow 0.$
\end{proof}

\begin{corollary}
Let $T$ be a positive compact operator on a Banach lattice $X$, and let
$\{S_n\}_{n=1}^{\infty}$ be positive operators on $X$ such that
$T\geq S_n\downarrow 0$.
If $X$ or $X^*$ has order continuous norm, then $\|S_n\|\downarrow 0.$
\end{corollary}

The proof of the following main theorem is inspired by the proof of  \cite[Theorem 3.3]{Schneeberger:71}.

\begin{theorem}\label{positive A B}
Suppose that the Banach lattice $X = \Pelc{X}$ is an order Pelczy\'nski decomposition of $X$, where $1\leq p\leq \infty$.
Let $Y$ be a Banach lattice with the property that there exist positive operators $S: Y \to X$ and $T: X \to  Y$
such that $\|S\| = 1$ and $T S$ is the identity operator on $Y$.
If $C$ is a positive operator on $(Y \oplus X)_p$, then $C$ is a commutator between two positive operators
in each of the following cases:
\begin{enumerate}
\item[(a)] $C$ is semi-compact, and $X$ and $X^*$ have order continuous norms.
\item [(b)] $C$ is semi-compact, $C^*$ is AM-compact and $X$ has order continuous norm.
\item [(c)] $C$ is AM-compact, $C^*$ is semi-compact and $X^*$ has order continuous norm.
\item [(d)] $T$ is compact, and either $X$ or $X^*$ have order continuous norms.
\end{enumerate}
\end{theorem}

\begin{proof}
We will simultaneously prove (a), (b) and (c). The case (d) then follows from (b) and (c), because compact operators
are both semi-compact and AM-compact.
Observe first that $1 = \|T S \| \le \|T\| \|S\| = \|T\|$.  Put $X_0 = Y$. The Banach lattice
$(Y \oplus X)_p =  \left(\displaystyle{\oplus_{n=0}^\infty} X_n\right)_p$ is
isometric and order isomorphic to $(Y \oplus X \oplus X \oplus X \oplus \cdots)_p$.
Under this isomorphism, the operator $C$ can be represented as
an infinite block operator matrix $[C_{ij}]_{i,j=0}^\infty$.

Choose a decreasing sequence $ \{\epsilon_n\}_{n=2}^\infty$ of positive real numbers such that the series
$\sum_{n=2}^\infty (2n+1) \epsilon_n$ converges. For each $n \ge 0$,
let $P_n: Y \oplus X \to Y \oplus X$ be the band projection on  the band
$\left(\displaystyle{\oplus_{k=n}^\infty} X_k\right)_p$.
In each of the cases (a), (b) and (c), the sequences $\{\|C P_n\|\}_{n=1}^\infty$  and $\{\|P_n C\|\}_{n=1}^\infty$
converge to $0$ by \Cref{monotoni normni pad}. Therefore, there exists a subsequence $\{P_{n_k}\}_{k=2}^\infty$
of the sequence $\{P_{n}\}_{n}$  such that $ \max\{\|C P_{n_k}\|, \|P_{n_k} C\|\} < \epsilon_k$ for all $k \ge 2$.
Put $n_1=0$. For $k=1, 2, \ldots$, the Banach lattice
$$ X^{(k)}=(X_{n_k+1} \oplus X_{n_k+2} \oplus \cdots \oplus X_{n_{k+1}})_p $$
is isometric and order isomorphic to $X$. It follows that
$X = \left(\displaystyle{\oplus_{k=1}^\infty} X^{(k)}\right)_p$
is also an order Pelczy\'nski decomposition of $X$. Therefore, without loss of generality
we can assume from the beginning that  $\max\{\|C P_n\| , \|P_n C\|\} < \epsilon_n$ for all $n \ge 2$.
Then the entries of the matrix $[\|C_{ij}\|]_{i,j=0}^\infty$ are dominated by the entries of the matrix
$$\begin{bmatrix}
\|C_{00}\| & \|C_{01}\| &\epsilon_2 & \epsilon_3 &\hdots\\
\|C_{10}\| & \|C_{11}\| &\epsilon_2 & \epsilon_3 &\hdots\\
\epsilon_2 & \epsilon_2 &\epsilon_2 & \epsilon_3 &\hdots\\
\epsilon_3 & \epsilon_3 &\epsilon_3 & \epsilon_3 &\hdots\\
\vdots     &\vdots     & \vdots     & \vdots     & \ddots
\end{bmatrix}.$$
Indeed, if $\max\{i, j\} \ge 2$, then
$$ \|C_{ij}\| \le  \|P_i C P_j\| \le \min\{ \|P_i C \|, \|C P_j\|\} <  \epsilon_{\max\{i,j\}} \ . $$

Now we define an infinite matrix $B$ by
$$ B = \begin{bmatrix}
0 & 0 & 0 & 0 & 0 & \hdots\\
S & 0 & 0 & 0 & 0 & \hdots\\
0 & I  & 0 & 0 & 0 & \hdots\\
0 & 0  & I & 0  & 0 & \hdots\\
0 & 0  & 0 & I  & 0 & \hdots\\
\vdots & \vdots & \vdots & \vdots & \vdots  & \ddots
\end{bmatrix} , $$
and a matrix $A=[A_{ij}]_{i,j=0}^\infty$ in the following way:
$$\begin{array}{lcc}
A_{i0}=0, & &i\geq 0\\
A_{i 1} = C_{i0} \, T, & &i\geq 0\\
A_{0 j} = C_{0, j-1}, & &j \geq 2\\
A_{ij}=\sum\limits_{k=1}^i C_{i-k+1,j-k}+S \, C_{0,j-i-1}, ,& &1 \leq i<j \\
A_{ij}=\sum\limits_{k=1}^{j-1} C_{i-k+1,j-k}+ C_{i-j+1,0} \, T, & & i\geq j \geq 2.
\end{array}$$
A direct calculation shows the formal equality $C=AB-BA$. Clearly, $B$ defines a bounded operator on  $(Y \oplus X)_p$ that is also positive. It remains to show the same for the operator $A$.
In order to prove this, we introduce the sequence $\{c_i\}_{i=-\infty}^\infty$ as follows:
$$ \begin{array}{lcc}
c_i=\sum\limits_{k=0}^\infty \|C_{k,k+i-1}\|, && i \geq 1,\\
c_i=\sum\limits_{k=0}^\infty \|C_{k-i+1,k}\|, && i \leq 1.
\end{array}$$
We also define the matrix $U=[u_{ij}]_{i,j=0}^\infty$ by $u_{ij}=c_{j-i}$.
Then $U$ is constant on each diagonal, and since
$$  \sum_{i=-\infty}^\infty c_i=\sum_{i,j=0}^\infty \|C_{ij}\| \leq
\sum_{i,j=0}^1 \|C_{ij}\| +\sum_{n=2}^\infty (2 n+1)\epsilon_n<\infty, $$
$U$ defines a bounded operator on both $l^1$ and $l^\infty$.
By the Riesz-Thorin interpolation theorem (see, e.g., \cite{Davies:90}), $U$ induces a bounded operator on $l^p$.
We now claim that $\|A_{ij}\| \leq \|T\| \, u_{ij}$ for all $i, j \ge 0$.
For $i \geq 0$, we have $\|A_{i 1}\| = \|C_{i, 0}\| \|T\|  \le \|T\| c_{1-i} = \|T\| u_{i1}$.
For $j \geq 2$, we have $\|A_{0 j}\| = \|C_{0, j-1}\| \le c_j = u_{0j}$.
If $1 \leq i<j$, then
$$ \|A_{ij}\| \leq \sum\limits_{k=1}^{i+1} \|C_{i-k+1,j-k}\| =
    \sum\limits_{m=0}^{i} \|C_{m, m+j-i-1}\| \leq c_{j-i} =u_{ij} . $$
If $i\geq j \geq 2$, then
$$ \|A_{ij}\| \leq \|T\| \, \sum\limits_{k=1}^{j} \| C_{i-k+1,j-k}\|  = $$
$$ =  \|T\| \, \sum\limits_{m=0}^{j-1} \| C_{m+i-j+1,m}\|  \leq \|T\| \, c_{j-i} =  \|T\| \, u_{ij} . $$
This completes the proof of the claim.

Suppose first $1\leq p<\infty$.
For  $x=(x_0, x_1, x_2, \ldots) \in  \left(\displaystyle{\oplus_{n=0}^\infty} X_n\right)_p$,
we have
$$ \|Ax\|^p=\sum_{i=0}^\infty \left\| \sum_{j=0}^\infty A_{ij} x_j \right\|^p
\leq \sum_{i=0}^\infty \left(\sum_{j=0}^\infty  \|A_{ij}\|\, \|x_j\| \right)^p \leq $$
$$  \leq \|T\|^p \cdot \sum_{i=0}^\infty \left(\sum_{j=0}^\infty u_{ij}\,\|x_j\|\right)^p \leq
 \|T\|^p \cdot \| U (\|x_0\|, \|x_1\|, \|x_2\|, \ldots)\|_p^p \leq $$
$$ \leq  \|T\|^p \cdot \|U\|^p \sum_{j=0}^\infty \|x_j\|^p =  (\|T\| \, \|U\| \, \|x\|)^p . $$
This shows that $A$ defines a bounded operator on the Banach lattice $(Y \oplus X)_p$ that is clearly positive.

Suppose now $p=\infty$. For $x=(x_0,x_1,x_2,\ldots)\in \left(\displaystyle{\oplus_{n=0}^\infty} X_n\right)_\infty$, we have
\begin{align*}
\|Ax\|_\infty&=\sup_{i\in \mathbb N_0}\left\|\sum_{j=0}^\infty A_{ij}x_j \right\|\leq \sup_{i\in\mathbb N_0} \sum_{j=0}^\infty \|A_{ij}\|\, \|x_j\|\\
 &\leq  \|x\| \, \sup_{i\in\mathbb N_0} \sum_{j=0}^\infty \|A_{ij}\|
\leq \|x\|\, \|T\|\sup_{i\in\mathbb N_0} \,\sum_{j=0}^\infty u_{ij}\\
&\leq \|x\|\,\|T\|\,\sum_{i=-\infty}^\infty c_i.
\end{align*}
This proves that $A$ defines a bounded operator on the Banach lattice $(Y\oplus X)_\infty$ which is again clearly positive.
\end{proof}


In the following example we will apply the fact that order intervals in atomic Banach lattices with order continuous norms are always compact by \cite[Theorem 6.1]{Wnuk:99}.

\begin{example}
(a) Consider the identity operator $I:c_0\to c_0$. Then $I$ is AM-compact, and $c_0$ and $c_0^* = l^1$ have order continuous norms. Since $I$ is not a commutator, this example shows that in \Cref{positive A B}(a) we cannot replace semi-compactness by AM-compactness.

(b)  Consider the Banach lattice $c$ of all convergent sequences. Then the Banach lattice dual of $c$ can be identified as $l^1\oplus_1 \mathbb R$ which is atomic and has order continuous norm (see, e.g., \cite[Theorem 16.14]{ABo:06}).
The identity operator $I:c\to c$ is semi-compact, its adjoint is AM-compact, yet $I$ is not a commutator. This example shows that in \Cref{positive A B}(b) we cannot remove order continuity of the norm of $X$.

(c) The identity operator $I:l^1\to l^1$ is AM-compact, its adjoint operator $I:l^\infty\to l^\infty$ is semi-compact, yet
the $I:l^1\to l^1$ is not a commutator. This example shows that in \Cref{positive A B}(c) we cannot omit the assumption that the norm of $X^*$ is order continuous.
\end{example}

\begin{corollary}
Let $C$ be a semi-compact positive operator on a separable infinite dimensional space
$L^p(M, \mu)$ with $1<p<\infty$. Then $C$ is a commutator between two positive operators.
\end{corollary}

\begin{proof}
By \cite[Theorem 7.1]{Bohnenblust:40}, the space $L^p(M, \mu)$  is isometric and order isomorphic to
one of the following canonical spaces: $l^p$, $L^p[0,1]$, $l^p \oplus L^p[0,1]$ or
$l^p_n \oplus L^p[0,1]$ for some $n \in \mathbb N$.  In the first three cases
Theorem \ref{positive A B} (a) can be applied immediately with $Y = X$ and $S = T = I$.
In the last case we must show that for $Y = l^p_n$ and  $X = L^p[0,1]$  the exist
positive operators $S: Y \to X$ and $T: X \to  Y$ such that $\|S\| = 1$ and
$T S$ is the identity operator on $Y$.

To this end, let $e_1$, $\ldots$, $e_n$ be the standard basis vectors of $l^p_n$, and let $\chi_i \in L^p[0,1]$ be the characteristic function of the interval $[\tfrac{i-1}{n}, \tfrac{i}{n})$, for $i=1, 2, \ldots, n$.
Define the positive operator $S: l^p_n \to L^p[0,1]$ by $S(e_i) = n^{1/p} \, \chi_i$ for $i=1, 2, \ldots, n$, and the positive operator $T: L^p[0,1] \to l^p_n$  by
$$ T f = n^{\frac{p-1}{p}} \cdot \sum_{i=1}^n e_i \cdot \int_{\frac{i-1}{n}}^{\frac{1}{n}} \! f (x) \, dx . $$
Then $\|S\| = 1$ and  $T S$ is the identity operator on $l^p_n$.
\end{proof}

Using Theorem \ref{positive A B} (d) in the preceding proof, we obtain the corresponding result for operators on $L^1$-spaces.

\begin{corollary}
Let $C$ be a compact positive operator on a separable infinite dimensional space
$L^1(M, \mu)$. Then $C$ is a commutator between two positive operators.
\end{corollary}



\end{document}